\documentclass[12pt]{article}

\usepackage{amsmath, amssymb, amsthm, amscd}

\usepackage{graphicx}

\usepackage{psfrag}

\usepackage{epstopdf}

\usepackage{url}

 \newtheorem{theorem}{\sc Theorem}[section]

 \newtheorem{lemma}[theorem]{\sc Lemma}

 \newtheorem{corollary}[theorem]{\sc Corollary}

 \newtheorem{proposition}[theorem]{\sc Proposition}

 \newtheorem{prop}[theorem]{\sc Proposition}

\newtheorem{defi}[theorem]{\sc Definition}

\newtheorem{result}[theorem]{\sc Result}

\newtheorem{definition}[theorem]{\sc Definition}

\newcommand{\cP}{\mathcal{P}}

\newcommand{\cL}{\mathcal{L}}

\newcommand{\cB}{\mathcal{B}}

\newcommand{\cC}{\mathcal{C}}

\newcommand{\cK}{\mathcal{K}}

\newcommand{\cX}{\mathcal{X}}

\newcommand{\PS}{\cP_S}

\newcommand{\LS}{\cL_S}

\newcommand{\PG}{\mathrm{PG}}

\newcommand{\AG}{\mathrm{AG}}

\newcommand{\BG}{\mathrm{BG}}

\newcommand{\ind}{\mathrm{ind}}

\newcommand{\cut}[1]{}

\newcommand{\linf}{\ell_\infty}

\newcommand{\oS}{\overline{S}}

\newcommand{\Ps}{\mathcal{P}_S}

\newcommand{\Ls}{\mathcal{L}_S}

\newcommand\oPs{\overline{\Ps}}

\newcommand\oPi{\overline{\Pi}}

\newcommand{\NP}{\mathcal{P}_N}

\newcommand{\GQ}{\mathrm{GQ}}

 \makeatletter

 \@addtoreset{equation}{section}

 \makeatother

 \makeatletter

 \@addtoreset{table}{section}

 \makeatother

\usepackage{color}

\begin{document}

\title{On the metric dimension of affine planes, biaffine planes and generalized quadrangles}

\date{}

\author{Daniele Bartoli, Tam\'as H\'eger, Gy\"orgy Kiss, Marcella Tak\'ats}





\maketitle


\begin{abstract}

In this paper the metric dimension of (the incidence graphs of) particular partial linear spaces is considered. We prove that the metric dimension of an affine plane of order $q\geq13$ is $3q-4$ and describe all resolving sets of that size if $q\geq 23$. The metric dimension of a biaffine plane (also called a flag-type elliptic semiplane) of order $q\geq 4$ is shown to fall between $2q-2$ and $3q-6$, while for Desarguesian biaffine planes the lower bound is improved to $8q/3-7$ under $q\geq 7$, and to $3q-9\sqrt{q}$ under certain stronger restrictions on $q$. We determine the metric dimension of generalized quadrangles of order $(s,1)$, $s$ arbitrary. We derive that the metric dimension of generalized quadrangles of order $(q,q)$, $q\geq2$, is at least $\max\{6q-27,4q-7\}$, while for the classical generalized quadrangles $W(q)$ and $Q(4,q)$ it is at most $8q$.


\bigskip



\end{abstract}


\section{Introduction}

In this work we study the metric dimension of the incidence graph of specific finite 

°

point-line incidence geometries.

For a connected graph $G=(V,E)$ and $x,y\in V$, $d(x,y)$ denotes the distance of $x$ and $y$ (that is, the length of the shortest path joining $x$ and $y$).

\begin{definition}

Let $G=(V,E)$ be a connected graph. A vertex $v \in V$ is resolved by $S=\{v_1,\ldots,v_n\}\subset V$ if the ordered list $(d(v,v_1),d(v,v_2),\ldots,d(v,v_n))$ is unique. $S$ is a \emph{resolving set} for $G$ if it resolves all the elements of $V$. The \emph{metric dimension} $\mu(G)$ of $G$ is the size of a smallest resolving set for it. A \emph{metric basis} of $G$ is a resolving set for $G$ of size $\mu(G)$.

\end{definition}

Resolving sets of graphs have been studied since the mid '70s, and a lot of study has been carried out in distance-regular graphs. For more information about resolving sets and related topics we refer to \cite{Bailey-small, Bailey-DRG, BC} and the references therein. The current study has been motivated by the work of Bailey and can be regarded as a continuation of \cite{HT2012}, where the metric dimension of projective planes of order $q\geq 23$ is determined. This result was extended for $q\geq13$ in the recent thesis \cite{Szilard}, so we have the following.

\begin{theorem}[\cite{HT2012, Szilard}]\label{4q-4}

The metric dimension of any projective plane of order $q\geq13$ is $4q-4$.

\end{theorem}

Moreover, \cite[Figure 3]{HT2012} lists all metric bases for projective planes of order $q\geq 23$.

In Section \ref{sec:affine}, we use Theorem \ref{4q-4} to deduce that the metric dimension of an arbitrary affine plane of order $q\geq 13$ is $3q-4$, and we describe all metric bases if $q\geq 23$. In Section \ref{sec:biaffine}, we study the metric dimension of biaffine planes (elliptic semiplanes of flag type) of order $q$, which turns out to be between $2q-2$ and $3q-6$. We show that the metric dimension of many Desarguesian biaffine planes is $3q-o(q)$, and prove a general lower bound $8q/3-7$ if $q\geq 7$. Moreover, we provide some considerations supporting that, unlike in case of projective and affine planes, the metric dimension of a biaffine plane does rely on its finer structure, thus its exact value cannot be derived from the incidence axioms of biaffine planes only. Projective planes can be considered as generalized $n$-gons with $n=3$. These structures were introduced by Tits in 1959. For a brief description of generalized $n$-gons we refer to \cite[Chapter 13]{PT}. In the graph theoretical point of view, the incidence graph of a generalized $n$-gon is a bipartite graph with diameter $n$ and girth $2n$. For $n=4$, these objects are called \emph{generalized quadrangles} (GQ for short). Section \ref{sec:gq} is devoted to resolving sets of generalized quadrangles. The main result of the section is that a GQ of order $(q,q)$ has metric dimension at least $\max\{6q-27,4q-7\}$, while the metric dimension of the classical GQs $W(q)$ and $Q(4,q)$ is at most $8q$. Note that (the incidence graphs of) projective, affine and biaffine planes of order $q$ have roughly $2q^2$ vertices, while GQs of order $(q,q)$ have roughly $2q^3$.

\subsection{Notation and preliminaries}

In the sequel, $S=\Ps\cup\Ls$ always denotes a set of vertices in the incidence graph of a given partial linear space $\Pi$ (in particular, an affine, biaffine or projective plane, or a generalized quadrangle), where $\Ps$ is a point-set and $\Ls$ is a line-set of $\Pi$. When we use graph theoretical notions in the context of a partial linear space, it should always be interpreted in the corresponding incidence graph. If a point $P$ is incident with a line $\ell$, we say that $P$ \emph{blocks} $\ell$ and that $\ell$ \emph{covers} $P$. A \emph{blocking set} of a partial linear space is a set of points that blocks every line; dually, a \emph{covering set} is a set of lines that covers every point. A point $P$ is \emph{essential} for a blocking set $\cB$ if $\cB\setminus\{P\}$ is not a blocking set.

We apply the same notation as in \cite{HT2012}; in particular, $PQ$ indicates the lines through two points $P$ and $Q$; $[P]$ and $[\ell]$ denote the set of all the lines through $P$ and all the points on the line $\ell$, respectively; once $S=\Ps \cup \Ls$ is given, \emph{inner points} and \emph{inner lines} indicate points or lines in $S$, whereas \emph{outer points} and \emph{outer lines} refer to points and lines not in $S$; a line $\ell$ is \emph{skew}, \emph{tangent} or a $t$-secant to $S$ if $[\ell]\cap \Ps$ is empty, just one point or has exactly $t$ elements, respectively; a point is \emph{covered} if it lies on at least one line of $\Ls$, \emph{uncovered} otherwise, and it is $t$-covered if it lies on exactly $t$ lines of $\Ls$.
If a vertex is in $S$ then it is trivially resolved by $S$. Also, the following lemma \cite[Lemma 6]{HT2012} clearly holds for arbitrary partial linear spaces.

\begin{lemma}[\cite{HT2012}]

A line $\ell$ which intersects $\Ps$ in at least two points is resolved by $S$. If a point lies on at least two inner lines then it is resolved by $S$.

\end{lemma}

We define a biaffine plane of order $q$ as an affine plane of order $q$ with a parallel class of lines removed. Biaffine planes are also called flag-type elliptic semiplanes. For further information, we refer the reader to \cite[Section 7.4, in particular point 13]{Dembowski}. Let $\Pi$ be a biaffine or an affine plane of order $q$. Then $\Pi$ can be uniquely embedded into a projective plane of order $q$ which we will denote by $\oPi$. Let $\linf$ be the unique line in $\oPi$ that has no points in $\Pi$. We call $\linf$ the line at infinity or ideal line. The \emph{direction of a line} of $\Pi$ is its intersection with $\linf$ in $\oPi$; thus the points of $\linf$ correspond to the parallel classes of $\Pi$ and thus will be called \emph{directions}. In the sequel, when working with an affine or biaffine plane $\Pi$, we consider it as embedded in $\oPi$ and use the respective notation without further mention. If $\Pi$ is a biaffine plane then a parallel class of lines, say, the class of vertical lines, is missing from $\Pi$; in $\oPi$, we denote the corresponding direction on $\linf$ by $(\infty)$, but $(\infty)$ is not considered as a direction for the biaffine plane; and we call the sets of $q$ pairwise non-adjacent points in $\Pi$ (corresponding to the point-sets of the $q$ vertical lines of $\oPi$) \emph{non-adjacency classes}. For the sake of completeness, we recall the basic combinatorial properties of biaffine planes.

A biaffine plane of order $q$ has $q^2$ points and $q^2$ lines; each point is incident with $q$ lines, each line is incident with $q$ points; for a non-incident point-line pair $(P,\ell)$, there exists exactly one line through $P$ not intersecting $\ell$, and there is exactly one point $Q$ on $\ell$ not collinear with $P$. There are $q$ parallel classes and $q$ non-adjacency classes, each containing $q$ elements (lines or points, respectively) and partitioning the line set and the point set of the plane, respectively. For a point $P$ or a line $\ell$, $C(P)$ and $C(\ell)$ will denote the non-adjacency or parallel class containing $P$ or $\ell$.

Let $S=\Ps\cup\Ls$ be a vertex set of the affine or biaffine plane $\Pi$. If $\Ls$ contains a line with direction $d$ then we call $d$ a \emph{covered direction} (with respect to $S$). A direction not covered by $S$ is called an \emph{uncovered} direction. If $\Pi$ is a biaffine plane then by a \emph{blocked} or \emph{unblocked class} (with respect to $S$) we mean a non-adjacency class that contains at least one or no point of $\PS$, respectively.

Finite generalized quadrangles can be defined in the following alternative way. 

\begin{definition}

Let $s$ and $t$ be positive integers. 

A point-line incidence geometry $\mathcal{G}=(\mathcal{P}, \mathcal{L}, \mathrm{I})$

is a \emph{generalized quadrangle of order $(s,t)$} if it satisfies the following axioms.

\begin{description}

\item[\bf (GQ1)]

Each point is incident with $t+1$ lines and two distinct points are incident with at most one line.

\item[\bf (GQ2)] 

Each line is incident with $s+1$ points and two distinct lines are incident with at most one point.

\item[\bf (GQ3)]

If $(P,\ell )\subset \mathcal{P}\times \mathcal{L}$ is a non-incident point-line pair then there is 

a unique pair $(P',\ell ')\subset \mathcal{P}\times \mathcal{L}$ for which

$P\, \mathrm{I}\, \ell '\, \mathrm{I}\, P'\, \mathrm{I}\, \ell $.

\end{description}

\end{definition}

From this definition it is easy to derive the basic combinatorial properties of GQs.

Let $\mathcal{G}$ be a generalized quadrangle of order $(s,t)$. Then

\begin{itemize}

\item

each point is collinear with $(t+1)s$ other points and each line is concurrent with $(t+1)s$ other lines;

\item

$\mathcal{G}$ contains $v=(s+1)(st+1)$ points and $b=(t+1)(st+1)$ lines; 

\item

if $P$ and $R$ are two non-collinear points of $\mathcal{G}$ then there are

$t+1$ points in $\mathcal{G}$ which are collinear with both $P$ and $R$;

\item

if $e$ and $f$ are two non-intersecting lines of $\mathcal{G}$ then there are $s+1$ lines in $\mathcal{G}$ which intersect both $e$ and $f$.

\end{itemize}

\section{Resolving sets for affine planes}\label{sec:affine}

It is easy to see that for a projective plane $\Pi$, a set $S$ is a resolving set if and only if for any two distinct outer lines $\ell$ and $\ell'$, $\Ps\cap[\ell]\neq\Ps\cap[\ell']$, and for any two distinct outer points $P$ and $P'$, $\Ls\cap[P]\neq\Ls\cap[P']$. Hence the next proposition is straightforward.

\begin{proposition}[\cite{HT2012}, Proposition 7] \label{resdef-p} 

$S=\Ps\cup\Ls$ is a resolving set for a finite projective plane if and only if the following properties hold for $S$:

\begin{description}

\item[\bf (P1)]{There is at most one outer line skew to $\Ps$.}

\item[\bf (P2)]{Through every inner point there is at most one outer line tangent to $\Ps$.}

\item[\bf (P1')]{There is at most one outer point not covered by $\Ls$.}

\item[\bf (P2')]{On every inner line there is at most one outer point that is $1$-covered by $\Ls$.}

\end{description}

\end{proposition}

The main difference between projective and affine planes is the existence of parallel lines. The distance between two lines in an affine plane can be either $2$ or $4$, depending on whether they intersect or not. This leads to the following modification of Proposition \ref{resdef-p}; the proof is left for the reader.

\begin{proposition}\label{resdef-a}

A set $S=\Ps \cup \Ls$ is a resolving set for an affine plane $\Pi$ if and only if the following hold.

\begin{description}

\item[\bf (A1)] There is at most one uncovered outer point.

\item[\bf (A2)] On every inner line, there is at most one $1$-covered outer point.

\item[\bf (A1')] For each covered direction $d$, there is at most one outer skew line with direction $d$. There is at most one outer skew line having an uncovered direction.

\item[\bf (A2')] For each inner point, there is at most one tangent line having uncovered direction.

\end{description}

\end{proposition}

We remark that the tangents in Proposition \ref{resdef-a} (A2') are necessarily outer and, in particular, all tangent lines with a covered direction are resolved. Furthermore, if there is at most one uncovered direction then (A2') is automatically satisfied, and (A1') simplifies to `there is at most one outer skew line in each direction'.

\begin{proposition}\label{embed}

Let $S=\PS\cup\LS$ be a resolving set for the affine plane $\Pi$, and suppose that there is a direction $P\in\linf$ that contains at least two lines of $\Ls$. Let $\oPs=\Ps\cup([\linf]\setminus\{P\})$. Then $\oS=(\oPs,\Ls)$ is a resolving set for $\oPi$.

\end{proposition}

\begin{proof}

Relying on Proposition \ref{resdef-p}, we check the four properties of resolving sets for projective planes for $(\oPs,\Ls)$.\\

(P1): {A skew line to $\oPs$ intersects $\linf$ in $P$. As $P$ is a covered direction, there is at most one outer skew line to $\Ps$ through $P$ by (A1').}\\

(P1'): {On $\linf$, the only outer point is $P$, which is covered by $\Ls$. In $\Pi$, there is at most one outer point not covered by $\Ls$ by (A1).}\\

(P2): {Let $Z\in\oPs$. We have to show that there is at most one outer tangent line to $\oPs$ through $Z$. If $Z\in\Ps$ then every line through $Z$ intersects $\linf$ in an inner point of $\oS$ with the only exception $ZP$. If $Z\in\linf$ is a covered direction then a tangent to $\oPs$ through $Z$ is skew to $\Ps$. By (A1'), there is at most one outer skew line to $\Ps$ through $Z$. Suppose now that $Z\in\linf$ is an uncovered direction. By (A1'), there is at most one outer skew line to $\Ps$ intersecting $\linf$ in an uncovered direction, so we are done.}\\

(P2'): {Let $\ell\in\Ls$. We have to show that there is at most one outer point on $\ell$ that is $1$-covered by $\Ls$. By (A2), there is at most one outer $1$-covered point on $\ell$ in $\Pi$. The points of $\linf$ are all inner points except $P$, which is covered by at least two lines of $\Ls$.}

\end{proof}

\begin{proposition}\label{Lbig}

Let $S=\PS\cup\LS$ be a resolving set for an arbitrary affine plane $\Pi$ of order $q$. If $|S|\leq 3q-4$ then $|\Ls|\geq 2q-3$.

\end{proposition}

\begin{proof}

Let $t$ be the number of $1$-covered outer points. Then $t\leq |\Ls|$ by (A2). Using (A1) and double counting on the size of $\Gamma=\{(\ell,P) \mid \ell \in \Ls, \ell \in [P], |[P] \cap \Ls|\geq 2\}$, we see that 

\[ |\Ls|q-t \geq |\Gamma| \geq 2(q^2-1-t-|\Ps|).\]

Then

\[|\Ls|(q-1)\geq 2(q^2-1-t-|\Ps|)+t-|\Ls|=2(q^2-1)-2(|\Ls|+|\Ps|)-t+|\Ls|\geq 2q^2-6q+6,\]

hence \[|\Ls|\geq \frac{2q^2-6q+4}{q-1} = 2q-4+\frac{2}{q-1}>2q-4.\]

\end{proof}

\begin{theorem}

Let $\Pi$ be an arbitrary affine plane of order $q\geq13$. Then the metric dimension of $\Pi$ is $3q-4$.

\end{theorem}

\begin{proof}

Let $S=\Ps\cup\Ls$ be a resolving set for $\Pi$ of size at most $3q-4$. 
Then, by Proposition \ref{Lbig}, $|\Ls|\geq 2q-3$. As there are $q+1$ directions and $q>4$ implies $2q-3>q+1$, we see that there is a parallel class that contains at least two lines from $\Ls$. 
Let $P$ be a point of $\linf$ that is covered by at least two lines of $\Ls$. Thus Proposition \ref{embed} can be applied to see that $\oS$ is a resolving set for $\oPi$ of size at most $4q-4$; moreover, $\oS$

contains $q$ collinear points (on $\linf$). Thus, as the metric dimension of $\oPi$ is $4q-4$ (Theorem \ref{4q-4}), we see that $|\oS|=4q-4$, whence $|S|=3q-4$ follows. 

\end{proof}

\begin{theorem}
Let $\Pi$ be an arbitrary affine plane of order $q\geq23$. Then there are four types of metric bases for $\Pi$, listed in Figure \ref{affineressets}.
\end{theorem}

\begin{proof}
In \cite{HT2012}, the complete list\footnote{Figure 3 of \cite{HT2012} contains a few potentially unclear details; an upgraded version is now available attached to \cite{HT2012} or at ArXiv:1207.5469.} of resolving sets of size $4q-4$ in projective planes of order $q\geq 23$ with $|\Ps|\leq |\Ls|$ is given. Thus to find the candidates for $\oS$ in the list, we need to look for resolving sets with either $q$ collinear points, or with $q$ concurrent lines and $|\Ls|<|\Ps|$ (the duals of the latter examples contain $q$ collinear points but, as they have more points than lines, they are not listed).

Looking through the list of \cite{HT2012}, we see that $\oS$ is either (C3), (C5), (C29) with $Z=P$, $(C30)$ with $Z=Q$, or the dual of (C1) or (C2). Removing the line with $q$ collinear points and its inner points, (C5) is the only one which results in a set that is not a resolving set for $\Pi$ (it violates (A2)), while (C29) with $Z=P$ and (C30) with $Z=Q$ give the same construction. Thus we end up with four different resolving sets depicted in Figure \ref{affineressets}.
\end{proof}

\begin{figure}[!ht]

\psfrag{R}{$R$}\psfrag{R'}{$R'$}\psfrag{e}{$e$}\psfrag{l1}{$\ell_1$}\psfrag{l0}{$\ell_0$}

\includegraphics[width=\textwidth]{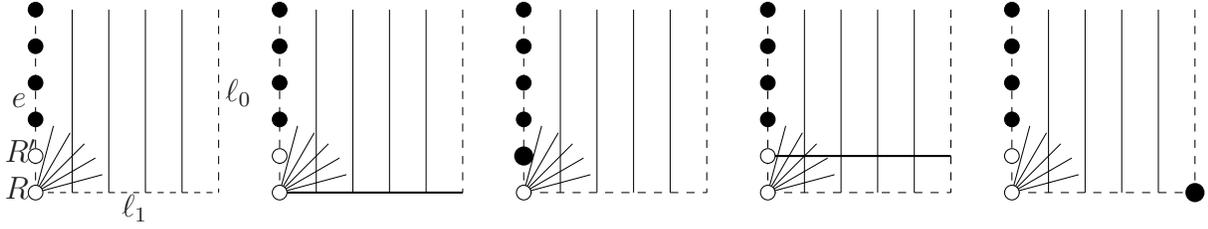}

\caption{\label{affineressets} \emph{The four types of smallest resolving sets for affine planes of order $q\geq 23$.}}
\footnotesize\mbox{}\\
Take an arbitrary parallel class $[P_\infty]$, say, the vertical lines, and choose two lines from it, $e$ and $\ell_0$. Let $R$, $R'$ be two arbitrary points on $e$, and let $e\neq\ell_1\in[R]$. Each smallest resolving set contains the following type of structure: $([e]\setminus\{R,R'\})\cup([P_\infty]\setminus\{e,\ell_0\})\cup([R]\setminus\{e,\ell_1\})$.

To obtain a resolving set seen in the figure, we have to add one of the following four elements, respectively: $\ell_1$; the line $\ell\in[R']$ parallel with $\ell_1$; $R'$; $\ell_0\cap\ell_1$.

\end{figure}

\section{On the metric dimension of biaffine planes}\label{sec:biaffine}

We denote an arbitrary biaffine plane of order $q$ by $B_q$, while $\BG(2,q)$ stands for the Desarguesian biaffine plane; that is, the biaffine plane obtained from the Desarguesian affine plane $\AG(2,q)$. Note that if $q\leq 8$ then the only biaffine plane of order $q$ is $\BG(2,q)$ (this follows from the well-known analogous fact for projective planes). The metric dimension of biaffine planes of small order, together with other small distance regular graphs, was determined by Bailey \cite{Bailey-small}. It turned out that $\mu(\BG(2,3))=4$, $\mu(\BG(2,4))=6$ and $\mu(\BG(2,5))=9$. In Bailey's work, $\BG(2,7)$ is exceptional as it is the only distance transitive graph on at most 100 vertices and valency between 5 and 13 whose metric dimension could not be calculated.

\subsection{General case}

First, we present some bounds using purely combinatorial tools, thus these results are valid for all biaffine planes.

\subsubsection{Bounds on $\mu(B_q)$}

In case of biaffine planes, Proposition \ref{resdef-p} needs further modifications because there are non-collinear points, too. The straightforward proof is omitted again.

\begin{prop}\label{baxioms}

$S=\PS\cup\LS$ is a resolving set for $B_q$ if and only if the following hold. 

\begin{description}

\item[\bf (B1)]{For each blocked class $C$, there is at most one uncovered outer point in $C$; furthermore, there is at most one outer uncovered point in the union of unblocked classes.}

\item[\bf (B2)]{On each inner line, there is at most one 1-covered point lying in an unblocked class.}

\item[\bf (B1')]{For each covered direction $d$, there is at most one skew outer line with direction $d$; furthermore, there is at most one outer skew line having an uncovered direction.}

\item[\bf (B2')]{On each inner point, there is at most one tangent line with uncovered direction.}

\end{description}

\end{prop}

Note that a 1-covered point in (B2) or a tangent line in (B2') is necessarily outer (as it is in an unblocked class, or has an uncovered direction). In particular, all 1-covered points lying in a blocked class and all tangent lines with a covered direction are resolved. Furthermore, if there is at most one uncovered direction then (B2') is automatically satisfied, and (B1') simplifies to `there is at most one outer skew line in each direction' and, dually, if there is at most one unblocked class then (B2) is automatically satisfied and (B1) simplifies to `there is at most one outer uncovered point in each non-adjacency class'.

\begin{proposition}\label{3q-6}

If $q\geq 4$ then $\mu(B_q)\leq 3q-6$.

\end{proposition}

\begin{proof}

We give a construction of a resolving set of size $3q-6$.

\begin{figure}[!ht]

\begin{center}

\psfrag{R}{$R$}\psfrag{P1}{$P_1$}\psfrag{l1}{$\ell_1$}\psfrag{T}{$T$}\psfrag{W}{$W$}\psfrag{Z}{$Z$}\psfrag{Pq}{$P_q$}\psfrag{X}{$X$}\psfrag{linf}{$\linf$}\psfrag{lq}{$\ell_q$}

\includegraphics[height=4.5cm]{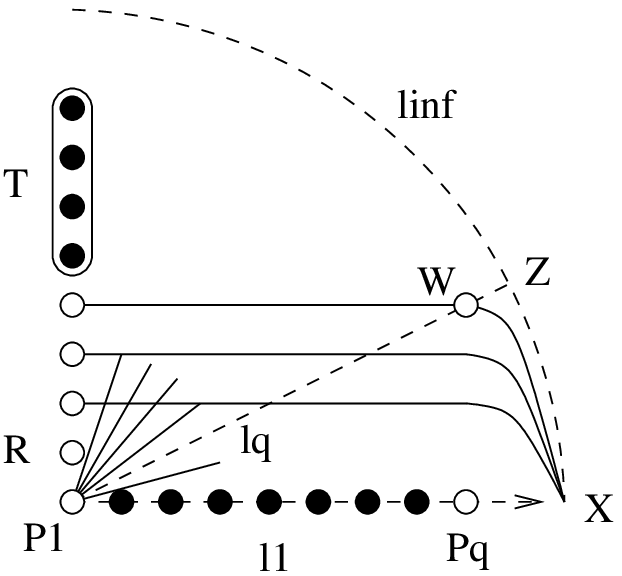}

\caption{A resolving set for $B_q$ of size $3q-6$. We require $ZP_q\cap C(P_1)\in T$.}

\end{center}

\end{figure}

Let $\ell_1=\{P_1,\ldots,P_q\}$, $[P_1]=\{\ell_1,\ldots,\ell_q\}$, $X=\ell_1\cap\linf$, $Z=\ell_q\cap\linf$, $W=C(P_q)\cap\ell_q$. Let $T\subset C(P_1)$ be a non-empty set of at most $q-3$ points such that $ZP_q\cap C(P_1)\in T$, $XW\cap C(P_1)\notin T$ (if $ZP_q\cap XW\in C(P_1)$, we may choose another numbering for the points of $\ell_1$), and let $R$ be an arbitrary point in $C(P_1)\setminus(T\cup \{P_1\})$ not covered by $XW$.

Let $\PS=T\cup \{P_2,\ldots,P_{q-1}\}$ and $\LS=\{\ell_2,\ldots,\ell_{q-1}\}\cup\{XQ\colon Q\in C(P_1)\setminus(T\cup\{P_1,R\})\}$. Then $|\PS|+|\LS|=3q-6$. We check the requirements of Proposition \ref{baxioms} to show that $S=\PS\cup\LS$ is a resolving set for $B_q$.

Clearly, $C(P_q)$ is the only unblocked class and $Z$ is the only uncovered direction; hence, (B2) and (B2') hold trivially. Through the direction $Z$, the only outer skew line is $\ell_q$ (as $ZP_q$ is blocked by a point of $T$); through $X$, the only outer skew line is $XR$; through any other direction $D$ the only possible outer skew line is $DP_q$. Therefore, (B1') is satisfied. Dually, in $C(P_q)$, the only outer uncovered point is $P_q$ (as $\ell_q\cap C(P_q)=W$ is covered); in $C(P_1)$, the only outer uncovered point is $R$; in any other class $C(P_i)$, the only possible outer uncovered point is $C(P_i)\cap\ell_q$. Hence (B1) is also satisfied.

\end{proof}

Note that in the above construction, the number of points may vary from $q-1$ to $2q-5$. Also note that Bailey's computations \cite{Bailey-small} show that $\mu(\BG(2,q))=3q-6$ for $q=4$ and $5$.

We proceed with some investigations on how small a resolving set for a biaffine plane may be. We will not obtain sharp results for the general case; in the Desarguesian case, stronger results will be obtained. As in case of affine resolving sets, one can prove a lower bound on the metric dimension of biaffine planes using that of projective planes; however, it is easy to obtain a lower bound directly.

\begin{proposition}\label{breslbs}

Let $S$ be a resolving set for $B_q$. Then $|\PS|\geq q-|S|/(q-1)$ and $|\LS|\geq q-|S|/(q-1)$.

\end{proposition}

\begin{proof}

The number of skew lines to $\PS$ is at most $|\LS|+q = |S|-|\Ps|+q$ by (B1'). As $k$ points in $B_q$ block at most $kq$ lines (with equality if and only if no two of them are collinear), we have $|\PS|q\geq q^2-(|S|-|\PS|+q)$, that is, $(q-1)|\PS|\geq q^2-q -|S|$, whence $|\PS|\geq q-|S|/(q-1)$ follows. Duality gives the statement for $\LS$.

\end{proof}

\begin{proposition}\label{blb}

For any biaffine plane $B_q$ of order $q$, we have $\mu(B_q)\geq 2q-2$.

\end{proposition}

\begin{proof}

Let $S=\PS\cup\LS$ be a resolving set for $B_q$. Suppose to the contrary that $|S|\leq 2q-3$. Then Proposition \ref{breslbs} gives $|\PS|\geq q-(2q-3)/(q-1)>q-2$, hence $|\PS|\geq q-1$. Similarly, $ |\LS|\geq q-1$, thus $|S|\geq 2q-2$, a contradiction.

\end{proof}

Next we give more detailed results on the sizes of $\PS$ and $\LS$ in terms of two parameters. These will not give a better lower bound immediately but we will make use of them later.

\begin{defi}

Let $S=\PS\cup\LS$ be a resolving set for $B_q$. Then $u=u(S)$ denotes the number of uncovered directions and $c=c(S)$ denotes the number of unblocked classes.

\end{defi}

\begin{prop}\label{deltamax}

Let $S=\PS\cup\LS$ be a resolving set for $B_q$. Then the number $\delta$ of skew lines to $\PS$ is at most $|\LS|+q$; moreover, $\delta\leq |\LS|+q-u+1$.

\end{prop}

\begin{proof}

The lines of $\LS$ may be skew. Regarding outer skew lines, Proposition \ref{baxioms} (B1') gives the required bound immediately.

\end{proof}

\begin{prop}\label{A8}\label{uncovered}

Let $S=\PS\cup\LS$ be a resolving set for $B_q$. Then

\begin{eqnarray}
|\PS|&\geq& \frac{u}{u+1}2(q-1), \label{numu}\\
|\LS|&\geq& \frac{c}{c+1}2(q-1). \label{numc}
\end{eqnarray}

\end{prop}

\begin{proof}

Let $P$ be an uncovered direction, $t_i(P)$ be the number of $i$-secant lines to $\PS$ through $P$, and $t_{\geq2}:=t_2+t_3+\dots+t_q$. Then $t_0+t_1+t_{\geq2}=q$. By Proposition \ref{baxioms} (B1'), $t_0\leq 1$. So we get $t_{\geq 2}\geq q-1-t_1$, hence $|\PS|\geq t_1+2t_{\geq2}\geq t_1+2(q-1-t_1)=2(q-1)-t_1$. Thus $t_1\geq 2(q-1)-|\PS|$, so on the $u$ uncovered directions we see at least $u(2(q-1)-|\PS|)$ tangents to $\PS$.

On the other hand, on a point of $\PS$, by Proposition \ref{baxioms} (B2'), there can be only one tangent with an uncovered direction; hence the total number of tangents with uncovered directions is at most $|\PS|$. These two give $|\PS|\geq u(2(q-1)-|\PS|)$, that is, $|\PS|\geq \frac{u}{u+1}2(q-1)$. Dually, $|\LS|\geq \frac{c}{c+1}2(q-1)$.

\end{proof}

Note that by duality, we may always assume that $|\PS|\leq|\LS|$.

\begin{lemma}\label{highlander}

Let $S=\PS\cup\LS$ be a resolving set for $B_q$ and suppose that $|\PS|\leq |\LS|$. If $|S| < 3(q-1)$ then $u\leq 2$. If $|S| < \frac83(q-1)$ then $u\leq 1$; if $|S|<\frac83(q-1)-1$ then $c\leq 4$.

\end{lemma}

\begin{proof}

By Proposition \ref{uncovered}, $|S|\geq 2|\PS|\geq \frac{u}{u+1}4(q-1)$, where the right-hand side is increasing in $u$. If $u\geq3$ then $|S|\geq 3(q-1)$ follows; if $u\geq2$ then $|S|\geq \frac83(q-1)$ holds.

By Proposition \ref{breslbs}, $|S|<3(q-1)$ implies $|\PS|\geq q-2$, hence $\frac{c}{c+1}2(q-1)\leq |\LS|\leq |S|-q+2$, thus $|S|\geq \left(\frac{2c}{c+1}+1\right)(q-1)-1$. Therefore, if $|S|<\frac83(q-1)-1$ then $\frac53>\frac{2c}{c+1}$ follows, a contradiction if $c\geq 5$.

\end{proof}

\subsubsection{Sharpness of the bounds}

It is natural to ask whether the upper bound $3q-6$ is sharp. It might be the case for Desarguesian biaffine planes; besides Bailey's computations for $q=4$ and $5$, we will show some results later which support this possibility. However, for general biaffine planes, we provide an imaginary construction which shows that one should be careful to think that Proposition \ref{3q-6} is sharp in general. To this end, we need the following notation. For a finite plane $\Pi$, let $\tau(\Pi)$ denote the size of the smallest blocking set in $\Pi$. Given a point $P$ and a line $\ell$ in a projective plane $\Pi_q$, let $\Pi_q\setminus[\ell]$ and $\Pi_q\setminus[P]$ denote the affine plane and the dual affine plane obtained by deleting $\ell$ and the points of $\ell$, or $P$ and the lines of $P$ from $\Pi_q$, respectively; and let $(\Pi_q\setminus[P])^*$ be the dual of $\Pi_q\setminus[P]$ (which is an affine plane).

Suppose now that $\cB$ is a blocking set in $\Pi_q\setminus[\ell]$ and $\cC$ is a covering set in $\Pi_q\setminus[P]$, and assume that $P\in \ell$. Then $B_{\ell,P}:=\Pi_q\setminus([\ell]\cup[P])$ is a biaffine plane in which $\cB\cup\cC$ is a resolving set (as there are no uncovered directions, nor unblocked classes, it is easy to verify this). Moreover, it is also easy to see that for any point $Q\in\cB$ and any line $r\in\cC$, $(\cB\setminus\{Q\})\cup(\cC\setminus\{r\})$ is also a resolving set for $B_{\ell,P}$; hence $\mu(B_{\ell,P})\leq \tau((\Pi_q\setminus[P])^*)+\tau(\Pi_q\setminus[\ell])-2$ follows.

Regarding the explicit size of this quasi-construction, the best lower bound known for the size of a blocking set in a general affine plane of order $q$ is $q+\sqrt{q}+1$ \cite{Bier, BT}, but its sharpness is wide open; the recent paper \cite{DBHSzVdV} shows that there are affine planes of order $q$ containing a blocking set of size at most $4q/3+5\sqrt{q}/3$. Thus the authors suspect that the metric dimension of some specific non-Desarguesian biaffine planes may be significantly smaller than $3q$. This construction idea does not work for Desarguesian biaffine planes as $\tau(\AG(2,q))=2q-1$ (Jamison \cite{Jamison}, Brouwer--Schrijver \cite{BS}).

Also, if we denote by $\tau^a_q$ the size of the smallest blocking set an affine plane of order $q$ may have, a general lower bound on $\mu(B_q)$ larger than $2\tau^a_q-2$ would imply that for a projective plane $\Pi_q$, $P\in\Pi_q$ and $\ell\in\Pi_q$, $P\in\ell$, both $\tau(\Pi_q\setminus[\ell])$ and $\tau(\Pi_q\setminus[P]^*)$ cannot be close to $\tau^a_q$. This would be a rather interesting phenomenon.

\subsection{The metric dimension of Desarguesian biaffine planes}

Now we turn our attention to Desarguesian biaffine planes where much better lower bounds than Proposition \ref{blb} can be obtained. As there cannot be too many lines in $\Pi_q$ not blocked by $\PS$, $\PS$ is almost a blocking set in $\Pi_q$. In such a situation one may apply stability results on blocking sets, which essentially say that if there are not too many skew lines to a point-set $X$ then $X$ can be extended to a blocking set by adding a few points to it. This motivates the following definitions.

\begin{defi}

A point-set $\cX$ of $\Pi_q$ is \emph{$k$-extendable} if $\cX$ can be extended to a blocking set of $\Pi_q$ by adding $k$ points of $\Pi_q$ to it. A set $\cK$ of $k$ points is a \emph{$k$-extender} of $\cX$ if $\cX\cup \cK$ is a blocking set. The set $\cX$ is \emph{$k$-punctured} if it is $k$-extendable, but not $(k-1)$-extendable.

\end{defi}

\begin{defi}

For a point $P\in\Pi_q$ and a point-set $\cX$, let the \emph{index} of $P$ with respect to $\cX$, $\ind_\cX(P)$, be the number of skew lines through $P$ to $\cX$.

\end{defi}

The following result is, in fact, equivalent with the formerly mentioned lower bound of Jamison and Brouwer--Schrijver for blocking sets in $\AG(2,q)$.

\begin{result}[Blokhuis--Brouwer \cite{BlBr}]\label{numtangents}

Let $\cB$ be a blocking set of $\PG(2,q)$. Then each essential point of $\cB$ is incident with at least $2q+1-|\cB|$ tangents to $\cB$.

\end{result}

\begin{lemma}\label{kext}

Let $\cK$ be a $k$-extender of a $k$-punctured point-set $\cX\subset \PG(2,q)$. If $P\in \PG(2,q)\setminus \cK$ then $\ind_\cX(P)\leq k$, and if $P\in \cK$ then $\ind_\cX(P)\geq 2q-|\cX|-k+1$.

\end{lemma}

\begin{proof}

Let $\cX^*=\cX\cup \cK$. Clearly, every skew line to $\cX$ contains a point of $\cK$, and each point of $\cK$ is essential for $\cX^*$. By Result \ref{numtangents}, there are at least $2q+1-|\cX^*|=2q-|\cX|-(k-1)$ tangents to $\cX^*$ through them. As these lines are all skew to $\cX$, we have $\ind_\cX(P)\geq 2q-|\cX|-k+1$ for all $P\in \cK$. Furthermore, for any point $P\in\PG(2,q)\setminus \cX^*$, at most $k$ lines of $[P]$ can be blocked by $\cK$; hence $\ind_\cX(P)\leq k$. Finally, for a point $P\in \cX$, $\ind_\cX(P)=0$.

\end{proof}

Now suppose that $S=\PS\cup\LS$ is a resolving set for $\BG(2,q)$; by $\PG(2,q)$ we denote the ambient Desarguesian projective plane.

\begin{prop}\label{notkext}

Suppose that $S=\PS\cup\LS$ is a resolving set for $\BG(2,q)$, $|S|\leq 3q-(k+u+3)$ and $c<2q-|\PS|-k$. Then $\PS\subset\PG(2,q)$ is not $k$-extendable.

\end{prop}

\begin{proof}

Suppose to the contrary that $K$ is a $k$-extender of $\PS$. We may assume (by choosing the smallest appropriate $k$) that $\PS$ is $k$-punctured. As $\PS\cap\linf=\emptyset$, there must be a point $P\in \linf\cap K$. Then, by Lemma \ref{kext}, $\ind_{\PS}(P)\geq 2q-|\PS|-k+1$. If $P=(\infty)$ then this means that the number $c$ of unblocked classes with respect to $S$ is at least $2q-|\PS|-k$, a contradiction. If $P\neq(\infty)$ then there are at least $2q-|\PS|-k$ biaffine skew lines to $\PS$ through $P$, among which there can be only one not in $\LS$ (Proposition \ref{baxioms} (B1')). As there are $q-1$ directions different from $P$ in $\BG(2,q)$, $u$ of which are uncovered, $|\LS|\geq (2q-|\PS|-k-1)+(q-1-u) = 3q-|\PS|-k-u-2$ follows, so $|S|=|\PS|+|\LS|\geq 3q-(k+u+2)$, a contradiction.

\end{proof}

Next we proceed by showing that if $S$ is small enough and $|\PS|\leq|\LS|$ then $\PS$ is indeed $k$-extendable for a suitable value of $k$ which, under certain conditions, will lead to lower bounds on the size of a biaffine resolving set. To show extendability, we rely on stability results on blocking sets of $\PG(2,q)$.

\begin{result}[Sz\H{o}nyi--Weiner \cite{stabblp}]\label{stblp}

Let $\cB$ be a set of points in $\PG(2,q)$, $q=p$ prime, with at most $\frac32(q+1)-\varepsilon$ points. Suppose that the number $\delta$ of skew lines to $\cB$ is less than $\left(\frac23(\varepsilon+1)\right)^2/2$. Then there is a line that contains at least $q-\frac{2\delta}{q+1}$ points of $\cB$.

\end{result}

As a line is a blocking set of $\PG(2,q)$ of size $q+1$, the above theroem claims that $\cB$ is $\left(\left\lfloor\frac{2\delta}{q+1}\right\rfloor+1\right)$-extendable. Note that if we set $\varepsilon=\frac32(q+1)-|\cB|$, the assumption $\delta<\left(\frac23(\varepsilon+1)\right)^2/2$ can be rephrased as $|\cB|<\frac32(q+1-\sqrt{2\delta})+1$.

\begin{result}[Sz\H{o}nyi--Weiner \cite{stabbl}]\label{stblq}

Let $\cB$ be a set of points in $\PG(2,q)$, $q=p^h$, $h\geq 2$. Denote the number of skew lines to $\cB$ by $\delta$ and suppose that $\delta\leq \frac{1}{100}pq$. Assume that $|\cB|<\frac32(q+1-\sqrt{2\delta})$. Then $\cB$ can be extended to a blocking set by adding at most

\[\frac{\delta}{2q+1-|\cB|} + \frac{1}{100}\]

points to it.

\end{result}

The next theorem shows that $\mu(\BG(2,q))\leq cq$ is not true in general for any constant $c<3$. Note that this means that there is no `generic' construction (that is, a construction relying on the axioms of biaffine planes only) of this size.

\begin{theorem}

Suppose that $S=\PS\cup\LS$ is a resolving set for $\BG(2,q)$, $q=p^h$, $p$ prime. Assume that $(\text{i}\,)$ $h=1$ and $q=p\geq17$, or $(\text{ii}\,)$ $h\geq2$ and $p\geq 400$. Then $|S| > 3q-9\sqrt{q}$.

\end{theorem}

\begin{proof}

Suppose to the contrary that $|S|\leq 3q-9\sqrt{q}$. Consider $S$ embedded into $\PG(2,q)$; thus, in the sequel, we have to take into account the non-adjacency classes and the line at infinity as lines. Note that as $\PS\neq\emptyset$, $c\leq q-1$. Propositions \ref{breslbs} and \ref{deltamax} yield $|\PS|\geq q-2$ and $\delta\leq |\LS|+q+c+1=|S|-|\PS|+2q\leq |S|+q+2 < 4q$.

By duality, we may assume $|\PS|\leq |\LS|$, and hence $|\PS|\leq|S|/2 < \frac32q-4\sqrt{q}<\frac32q-1$. $|S|/2 < \frac32(q+1-\sqrt{2|S|+2q+4})$ follows from the indirect assumption, hence $|\PS| < \frac32(q+1-\sqrt{2\delta})$. In case (ii), $\frac{1}{100}pq\geq4q>\delta$ also holds, thus we may use Results \ref{stblp} and \ref{stblq} to deduce that $\PS$ is $k$-extendable with $k=\left(\left\lfloor\frac{2\delta}{q+1}\right\rfloor+1\right)$ in case (i), and with $k=\left\lfloor\frac{\delta}{2q+1-|\PS|} + \frac{1}{100}\right\rfloor$ in case (ii). In both cases, $k\leq 8$.

By Lemma \ref{highlander}, $u\leq 2$, thus $|S|\leq 3q-13 \leq 3q-(k+u+3)$ holds. Then \eqref{numc} of Proposition \ref{uncovered} and $q>16$ give

\[\frac{c}{c+1}2q-2 < \frac{c}{c+1}2(q-1) \leq |\LS|=|S|-|\PS|\leq 2q-9\sqrt{q}+2\ < 2q-8\sqrt{q}-2,\]

thus

\[1-\frac{1}{c+1}\leq 1-\frac{4}{\sqrt{q}},\]

so $c\leq\sqrt{q}/4-1$. As $|\PS|\leq \frac32q-\frac92\sqrt{q}$ and $k\leq 8$, $\sqrt{q}/4-1\leq 2q-|\PS|-8\leq 2q-|\PS|-k$ follows. Therefore, by Proposition \ref{notkext}, $\PS$ is not $k$-extendable, a contradiction.

\end{proof}

Let us remark that the above proof could give a slightly better estimate than $|S|>3q-9\sqrt{q}$ but, as we do not think the result sharp, we decided to use simpler formulas. To give a more general but considerably weaker lower bound on $\mu(\BG(2,q))$, we need the following result conjectured by Metsch and proved by Sz\H{o}nyi and Weiner.

\begin{result}[Sz\H{o}nyi--Weiner {\cite[Theorem 4.1]{WSz2014}}] \label{mconj1}

Let $\cB$ be a point set in $\PG(2,q)$. Pick a point $P$ not from $\cB$ and assume that through $P$ there pass exactly $r$ lines meeting $\cB$ (that is containing at least $1$ point of $\cB$). Then the total number of lines meeting $\cB$ is at most \[1 + rq + (|\cB| - r)(q + 1 - r).\]

\end{result}

It is convenient to give an equivalent formulation of Result \ref{mconj1}.

\begin{result}\label{Metsch}

Let $\delta$ denote the number of skew lines to a point set $\cB$ in $\PG(2,q)$. Then for any point $P\notin \cB$, 

\begin{equation}
\ind_\cB(P)^2-(2q+1-|\cB|)\ind_\cB(P)+\delta\geq 0.  \label{metschineq}
\end{equation}

\end{result}

This quadratic inequality means that the index of a point is either small or large, which will be essential for the next proof.

\begin{proposition}\label{3ext}

Assume that $|S| < \frac83q-7$ and $|\PS|\leq|\LS|$. Then $\PS$ is $3$-extendable.

\end{proposition}

\begin{proof}

Lemma \ref{highlander} yields $u\leq1$ and $c\leq4$.

As in the previous proof, we use $|\PS|\geq q-2$, $\delta\leq |\LS|+q+c+1=|S|-|\PS|+q+c+1\leq |S|+7$ and $|\PS|\leq |S|/2\leq4q/3-\frac72$. Apply Result \ref{Metsch} with $\cB=\PS$, substitute $\ind_\cB(P)=4$ into \eqref{metschineq} and use $|\PS|\leq |S|/2$ and $\delta\leq|S|+7$ to obtain 

\[0\leq 16-4(2q+1-|\PS|)+\delta\leq 19-8q+3|S| \leq -2,\]

a contradiction. Thus neither $4$ nor $2q-|\PS|-3$ can be the index of a point of $\PG(2,q)$. Thus for every point $P\in\PG(2,q)$, $\ind_{\PS}(P)\leq 3$ or $\ind_{\PS}(P)\geq 2q-|\PS|-2$.

Let $\cK$ be the set of points with large index and $|\cK|=k$. Let $\ell$ be a skew line to $\PS$. If we had $\max_{P\in\ell}\{\ind_{\PS}(P)\}=m\leq 3$ then $\delta\leq 1+(m-1)(q+1)=(m-1)q+m$ and, by Result \ref{Metsch} and $|\PS|\leq \frac43q-\frac72$,
\[\begin{aligned}
0\leq m^2-(2q+1-|\PS|)m+(m-1)q+m &\leq m^2-\left(\frac23q+\frac92\right)m + (m-1)q+m\\
&=\left(\frac{m}{3}-1\right)q+m^2-\frac72m<0
\end{aligned}\]

followed, a contradiction. Thus every skew line contains a point with index greater than 3, thus at least $2q-|\PS|-2$. This means that $\PS\cup \cK$ is a blocking set for $\PG(2,q)$. Now it remains to check $k\leq3$. Suppose to the contrary that there exist four points with large index. When considering the $i$th point of these, we see at least $(2q-|\PS|-2)-(i-1)$ skew lines to $\PS$ not incident with the first $i-1$ points ($i=1,2,3,4$), and thus $8q/3 > |S|+7\geq\delta \geq 4(2q-|\PS|-2)-6 \geq 8q/3$, a contradiction. 

\end{proof}

\begin{theorem}\label{biaffin}

The metric dimension of $\BG(2,q)$ is at least $8q/3-7$.

\end{theorem}

\begin{proof}

If $q\leq 7$ then $2q-2\geq 8q/3-7$, thus Proposition \ref{blb} gives the result. Let $q\geq8$. Suppose to the contrary that there is a resolving set $S=\PS\cup\LS$ for $\BG(2,q)$ of size $|S| < 8q/3-7$. By duality we may assume that $|\PS|\leq|\LS|$. By Proposition \ref{3ext}, $\PS$ is $3$-extendable.

By Lemma \ref{highlander}, $u\leq 1$ and $c\leq 4$. Let $k=3$. Then $|S|\leq 3q-(k+u+3)$ immediately follows from $u\leq 1$ and $|S| < 8q/3-7$, and $c<2q-|\PS|-k$ follows from $|\PS|<4q/3-7/2$, $c\leq 4$ and $q\geq 8$. Thus we may apply Proposition \ref{notkext} to obtain that $\PS$ is not $3$-extendable, a contradiction.

\end{proof}

\section{On the metric dimension of generalized quadrangles}\label{sec:gq}

Generalized quadrangles are well-known and much studied objects in finite geometry, see \cite{PT} as for a comprehensive book in the topic. In this section we consider two particular types of quadrangles. A generalized quadrangle of order $(s,t)$ is denoted by $\GQ(s,t)$.

\subsection{Resolving sets for $\GQ(s,1)$}

The simplest generalized quadrangles have order $(s,1)$, or dually, $(1,t)$. These objects are called \emph{grids}, because the points of the unique $\GQ(s,1)$ form an $(s+1)\times(s+1)$ grid and the lines belong to two distinct parallel classes, each of them of size $s+1$ (corresponding to the rows and the coloumns of the grid).

In general, the grid graph $G_{n,m}$ has an $n$ by $m$ grid as its vertex set where two vertices are adjacent if and only if they are in the same row or coloumn. (Grid graphs are also called rook graphs.) Clearly, the problem of finding a subset consisting only of points in $\GQ(s,1)$ resolving all the remaining points is equivalent to finding a resolving set for $G_{s+1,s+1}$. This latter question has been addressed in \cite[Theorem 6.1]{CHMPPSW2007} in a more general way.

\begin{theorem}[C\'aceres et al.\ \cite{CHMPPSW2007}]\label{th:grids}

Let $G_{n,m}$ be an $n\times m$ grid, with $n\geq m\geq 1$. The metric dimension of $G_{n,m}$ is given by

\begin{equation}\label{eq:grid}
\mu(G_{n,m})=\left\{ \begin{array}{ll}
\left\lfloor \frac{2(n+m-1)}{3} \right \rfloor, & \mbox{ if } m\leq n\leq 2m-1,\\
n-1,& \mbox{ if } n\geq 2m.\\
\end{array}
\right.
\end{equation}

\end{theorem}

Note that $\mu(G_{n,m})\geq \left\lfloor \frac{2(n+m-1)}{3} \right \rfloor$ always holds. Using this theorem we can easily deduce the metric dimension of $\GQ(s,1)$.

\begin{corollary}

The metric dimension of $\GQ(s,1)$ is $\varphi(s)$, with

\[\varphi(s)= \left\{ \begin{array}{ll}

4r+1, & \mbox{ if } s=3r,\\

4r+2, & \mbox{ if } s=3r+1,\\

4r+3, & \mbox{ if } s=3r+2.\\

\end{array}\right.\]

\end{corollary}

\begin{proof}

First of all note that $\varphi(s)=\mu(G_{s+1,s+1})+1$ if $s\equiv0 \pmod{3}$ and $\varphi(s)=\mu(G_{s+1,s+1})$ otherwise. Let us construct the following resolving sets of $\GQ(s,1)$ to prove $\mu(\GQ(s,1))\leq\varphi(s)$. We identify the points of  $\GQ(s,1)$ with ordered pairs $(i,j)$, $1\leq i,j\leq s+1$, and the lines with $h_a = \{(i,a) : i \in\{1,\ldots,s+1\}\}$ and $v_a = \{(a,i) : i \in\{1,\ldots,s+1\}\}$.

Let $s=3r+t$ with $t\in\{0,1,2\}$, and consider 

\[\mathcal{S} = \{(1+3i,2+3i), (2+3i,1+3i), (1+3i,3+3i), (3+3i,1+3i)\: \ i =0,\ldots,r-1\}.\]
We define $S'$ in the following way.

\begin{enumerate}

\item If $t=0$ then $S':=\mathcal{S} \cup \{v_{s+1}\}$. 	 							

\item If $t=1$ then $S':=\mathcal{S} \cup \{(1+3r,1+3r),(2+3r,1+3r)\}$.					

\item If $t=2$ then $S':=\mathcal{S} \cup \{(1+3r,2+3r), (1+3r,3+3r)\}\cup\{v_{s+1}\}$.	

\end{enumerate}

It is easy to see that $S'$ is a resolving set for $\GQ(s,1)$ of size $4r+t+1$, so we omit the proof.

For a lower bound, suppose that a resolving set $S$ of $\GQ(s,1)$ contains $\alpha$ lines of the first parallel class and $\beta$ lines of the second one. The points of $\GQ(s,1)$ not contained in the $\alpha+\beta$ lines form a $(s+1-\alpha)\times(s+1-\beta)$ grid. In order to resolve this set of points, we need at least $\mu(G_{s+1-\alpha,s+1-\beta})\geq\left\lfloor \frac{2(s+1-\alpha+s+1-\beta-1)}{3} \right \rfloor=\left\lfloor \frac{4s+2 - 2(\alpha+\beta)}{3} \right \rfloor$ points, by Theorem \ref{th:grids}. Thus $|S|\geq \left\lfloor \frac{4s+2 - 2(\alpha+\beta)}{3} \right \rfloor+\alpha+\beta$, which is easily seen to be smaller than $\varphi(s)$ if and only if $\alpha+\beta=0$. Thus the assertions hold if $\alpha+\beta>0$. If $\alpha+\beta=0$ then $S$ consists entirely of points, thus corresponds to a resolving set for $G_{s+1,s+1}$, and so $|S|\geq \mu(G_{s+1,s+1})$. This shows that for $t=1,2$, $|S|\geq \varphi(s)$. Concerning the case $t=0$, note that $S$ resolves also the lines of $\GQ(s,1)$, so it cannot have two skew lines. Following the proof of \cite[Theorem 6.1]{CHMPPSW2007}, if $s=3r$, all resolving sets for $G_{s+1,s+1}$ of size $\mu(G_{s+1,s+1})=\varphi(s)-1$ have two skew lines, and thus do not resolve the lines of $\GQ(s,1)$. Therefore, $|S|\geq\varphi(s)$ holds in this case as well.

\end{proof}

\subsection{Resolving sets for $\GQ(q,q)$}

In this subsection an important class of generalized quadrangles is considered. This class contains two of the five classical quadrangles. We give a general lower bound first. If two points $P$ and $Q$ are adjacent, we write $P\sim Q$ and $P\not\sim Q$ otherwise.

\begin{proposition}\label{GQlb}

The metric dimension of any $\GQ(q,q)$ is at least $\max\{6q-27,4q-7\}$.

\end{proposition}

\begin{proof}

Let $\Pi$ be a $\GQ(q,q)$, $S=\Ps \cup \Ls$ be a resolving set for $\Pi$, $\NP$ denote the set of the uncovered outer points and let $|\Ps |=k$ and $|\Ls| = m$. By duality, we may assume that $m\geq k$, and hence $|S| \geq 2k$. As $\LS$ covers at most $m(q+1)$ points, $|\NP|\geq q^3+q^2+q+1-k-m(q+1)$. The distance between any point of $\NP$ and any line of $\Ls $ is $3$, so the elements of $\NP$ must be distinguished by their distances from the elements of $\Ps$. The distance of two points is either $2$ or $4$, according to whether the points are collinear or not. Let $C_i$ be the set of points of $\NP$ that are collinear with exactly $i$ inner points. As $|\Ps|=k$, it follows from the pigeonhole principle that $|C_i|\leq \binom{k}{i}$. 
First, let $\Gamma:=\{(P,Q)\colon P\in \NP\setminus(C_0\cup C_1), Q\in \Ps, P\sim Q\}$, and let us count $|\Gamma|$ in two different ways. On the one hand, we may choose $P$ arbitrarily from $\NP\setminus(C_0\cup C_1)$, and then we find at least two admissible pairs for $P$; thus 

\begin{equation}\label{est1}
\begin{aligned}
|\Gamma|&\geq \left( q^3+q^2+q+1-k-m(q+1) - |C_0|-|C_1| \right)\cdot2\\ &\geq 2(q^3+q^2+q)-2m(q+1) -2k-2|C_1|.
\end{aligned}
\end{equation}
On the other hand, each point $Q\in\PS$ is collinear with at most $(q+1)q$ points of $\NP$. Summing up these for all $Q\in\PS$, we obtain $k(q+1)q$, but we have also counted the points of $C_i$ exactly $i$ times, hence we get
\begin{equation}\label{est2}
|\Gamma|\leq k(q+1)q-|C_1|.
\end{equation}
Combining \eqref{est1} with \eqref{est2} and using $|C_1|\leq k$ and $m=|S|-k$ we obtain
\[k(q^2-q+1)+2|S|(q+1)-2(q^3+q^2+q)\geq 0.\]
Using $|S|\geq 2k$, one may deduce
\[|S|\geq \frac{4q(q^2+q+1)}{q^2+3q+5}=4q-7-\frac{q^2-5q-35}{q^2+3q+5} > 4q-8.\]

Next we modify and refine the above calculations to derive $|S|\geq 6q-27$. Let now $\Gamma:=\{(P,Q)\colon P\in \NP\setminus{(C_0\cup C_1\cup C_2)}, Q\in \Ps, P\sim Q\}$. Similarly as we obtained \eqref{est1}, we get

\begin{equation*}
|\Gamma|\geq \left( q^3+q^2+q+1-k-m(q+1) - |C_0|-|C_1|-|C_2| \right)\cdot3.
\end{equation*}

For a point $Q\in \Ps$, let $\NP(Q)$ be the number of points of $\NP$ adjacent to $Q$, and denote by $h$ the number of inner lines through $Q$. Then $|\NP(Q)|\leq (q+1-h)q$. If $h\geq 1$ then $|\NP(Q)|\leq q^2$. If $h=0$ then each line of $\Ls$ contains a point adjacent to $Q$, and that point is covered by at most $q$ inner lines, thus $|\NP(Q)|\leq (q+1)q-m/q$. If $m\geq q^2$ then $|S|\geq q^2+1\geq 6q-27$, so we are done; thus we may assume $m<q^2$, in which case $|\NP(Q)|\leq (q+1)q-m/q$ holds. Summing up these estimates for all points of $\PS$ we get $|\Gamma|\leq k(q+1)q-km/q$. In this estimate we have counted each point of $C_i$ exactly $i$ times, thus 

\begin{equation*}
|\Gamma|\leq k(q+1)q-km/q - |C_1|-2|C_2|
\end{equation*}

also holds. Rearranging, substituting $m=|S|-k$ and $|C_i|\leq \binom{k}{i}$, the two estimates together give

\begin{equation*}
\left(\frac12 + \frac1q\right)k^2-\frac{k|S|}{q}+\left( q^2-2q+\frac32\right)k + 3|S|(q+1)-3(q^3+q^2+q)\geq 0.
\end{equation*}

Suppose now to the contrary that $|S|\leq 6q-28$. Then, using $2k\leq |S|\leq 6q-28$ we obtain

\begin{equation*}
f(q,k):=\left(\frac12 - \frac1q\right)k^2+\left( q^2-2q+\frac32\right)k + 3(6q-28)(q+1)-3(q^3+q^2+q)\geq 0.
\end{equation*}
Recall that $0\leq k\leq |S|/2$. Thus, to get a contradiction, it is enough to see that $f(q,0)$ and $f(q,3q-14)$ are both negative by the convexity of $f(q,k)$ in $k$. As $f(q,0)=-3q^3+15q^2-69q-84<0$ and $f(q,3q-14)=-q^2/2-175q/2-196/q+77< 0$, we conclude that $|S|\geq 6q-27$ must hold.

\end{proof}

There are two known infinite series of classical generalized quadrangles of order $(q,q)$. To make the paper self-contained, we recall their brief descriptions.

In $\mathrm{PG}(3,q)$, all points of the space and the self-conjugate lines of a null polarity with the incidence inherited from $\mathrm{PG}(3,q)$ form a generalized quadrangle of order $(q,q)$.

This quadrangle is denoted by $\mathcal{W}(q)$. The lines of $\mathcal{W}(q)$ are lines in $\mathrm{PG}(3,q)$, hence each of them contains $q+1$ points, so axiom (GQ2) is satisfied with $s=q$.

If $\alpha $ is a null polarity of $\mathrm{PG}(3,q)$ then each point is autoconjugate. Let $P$ be an arbitrary point. If $\ell $ is a line through the points $P$ and $R$ then $\ell ^{\alpha }$ is the intersection of the planes $P^{\alpha }$ and $R^{\alpha }$. Hence $\ell $ is self-conjugate if and only if $P\in R^{\alpha }$ (and $R\in P^{\alpha }$).

As $P\in P^{\alpha }$ also holds, this means that the self-conjugate lines through $P$ are the lines of the pencil with carrier $P$ in the plane $P^{\alpha }$. Each pencil contains $q+1$ lines hence $\mathcal{W}(q)$ satisfies axiom (GQ1) with $t=q$.

Let $(P,\ell )$ be a non-incident point-line pair in $\mathcal{W}(q)$. Then $\ell $ is a self-conjugate line of $\alpha $ but it does not contain $P$, hence $\ell $ is not contained in the plane $P^{\alpha}$.  

So $\ell $ intersects $P^{\alpha }$ in a unique point, say $P'$ in $\mathrm{PG}(3,q)$. Then $PP'$ is the unique self-conjugate line through $P$ which meets $\ell $, hence axiom (GQ3) is satisfied by $\mathcal{W}(q)$.

Let $\mathcal{Q}$ be a non-singular quadric in $\mathrm{PG}(4,q)$.  Then the points of $\mathcal{Q}$ and the lines contained in $\mathcal{Q}$ with the incidence inherited from $\mathrm{PG}(4,q)$ form a generalized quadrangle of order $(q,q)$. This quadrangle is denoted by $\mathcal{Q}(4,q)$. It is well-known that $Q(4,q)$ is isomorphic to the dual of $W(q)$ and, if $q$ is even then $W(q)$ is self-dual. Hence, from the graph theoretic point of view, the incidence graphs of $W(q)$ and $Q(4,q)$ are always isomorphic.

\begin{definition}

A triple $(x,y,z)$ of points is called a \emph{triad} if no two of them are collinear. A point $u$ is a \emph{center} of a triad if $u$ is collinear with each of the three points of the triad. 
\end{definition}

By 1.3.6.\ (iii) and 3.3.1. of \cite{PT}, we have the following.

\begin{proposition}\label{triad}

If $q$ is odd then every triad of $Q(4,q)$ has either 0 or 2 centers.

\end{proposition}

We proceed by constructing small resolving sets for $W(q)$ as the union of two semi-resolving sets.

\begin{proposition}\label{srsWqpt}

There exists a semi-resolving set of size $4q$ for the points of $W(q)$.

\end{proposition}

\begin{proof}

Embed $W(q)$ into $\mathrm{PG}(3,q)$, let $\pi $ be the null-polarity whose self-conjugate lines are the lines of $W(q)$. First we construct a semi-resolving set for the points of $W(q)$. This construction does not depend on the parity of $q$. Let $a_1,\, a_2$ and $a_3$ be three pairwise skew lines of $W(q)$. These lines define a hyperbolic quadric $\mathcal{H}$ in $\mathrm{PG}(3,q)$. Let $a_4$ be a line of $W(q)$ which has empty intersection with $\mathcal{H}$ (an easy calculations shows the existence of such a line). We claim that the set $\Ps =[a_1]\cup [a_2]\cup [a_3] \cup [a_4]$ is a semi-resolving set of size $4q+4$ for the points of $W(q)$.

Let $C$ be an outer point. Then the plane $C^{\pi}$ contains none of the lines $a_1,\, a_2,\, a_3$ and $a_4$, because these are autoconjugate lines. For $i=1,2,3,4$ let $A_i=C^{\pi}\cap a_i$. As $W(q)$ does not contain any triangle, this means that $d(C,A_i)=2$ and $d(C,P)=4$ for all $P\in \Ps \setminus \{A_1,A_2,A_3,A_4\}$. First, suppose that $A_1,A_2$ and $A_3$ are not collinear in $\mathrm{PG}(3,q)$. Then the intersection of the three planes $A_1^{\pi}$, $A_2^{\pi}$ and $A_3^{\pi}$ in $\mathrm{PG}(3,q)$ is the single point $C$, hence $C$ is resolved in this case. If $A_1,A_2$ and $A_3$ are collinear in $\mathrm{PG}(3,q)$ then let $e$ denote the line joining them. Now the intersection of the three planes $A_1^{\pi}$, $A_2^{\pi}$ and $A_3^{\pi}$ in $\mathrm{PG}(3,q)$ is the line $e^{\pi}$. As $a_i\subset A_1^{\pi}$, the lines $e^{\pi}$ and  $a_i$ are coplanar, hence $e^{\pi}\cap a_i\neq \emptyset $ for $i=1,2,3$. The lines $a_1,\, a_2$ and $a_3$ are pairwise skew, hence $\mathcal{H}$ contains three distinct points of $e^{\pi}$, which means that the whole line $e^{\pi}$ is contained in $\mathcal{H}$. As $a_4$ has empty intersection with

$\mathcal{H}$, it has empty intersection with $e^{\pi}$, too. This means that if $E$ is any point on $e^{\pi}$ then the plane $E^{\pi}$ does not contain the line $a_4$, so $E^{\pi}\cap a_4$ is a unique point. Hence

the points of $e^{\pi}$ are also resolved by $\Ps$.

It is easy to see that if we delete one point from each of the lines $a_1,\, a_2,\, a_3$ and $a_4$ then the remaining $4q$ points also form a semi-resolving set for the points of $W(q)$.

\end{proof}

\begin{corollary}

If $q$ is even then the metric dimension of $W(q)$ is at most $8q$.

\end{corollary}

\begin{proof}

If $q$ is even then $W(q)$ is self-dual, hence the dual of a semi-resolving set for the points is a semi-resolving set for the lines. Thus the union of the semi-resolving set constructed in the previous proposition and its dual is a resolving set for $W(q)$ and its size is $8q$.

\end{proof}

\begin{proposition}\label{srsWqln}

If $q$ is odd then there is a semi-resolving set of size $5q-4$ for the lines of $W(q)$, which contains exactly $q-3$ points, all incident with the same line.

\end{proposition}

\begin{proof}

We may look for a semi-resolving set for the points of $Q(4,q)$ instead. Let $U$ be an arbitrary point, and suppose that $\ell_0,\ell_1,\ldots,\ell_q$ are the $q+1$ lines incident with $U$. Let $U\neq W\in\ell_0$, let $\ell$ be any line through $W$ different from $\ell_0$. We claim that $\Ps=[\ell_0]\cup[\ell_1]\cup[\ell_2]\cup[\ell]\setminus\{U,W\}$ and $\Ls=\{\ell_4,\ldots,\ell_q\}$ form a semi-resolving set $S$ for the points. $U$ and $W$ are clearly resolved by their distances from the points of $\ell_1$ and $\ell$. Let $T\notin\Ps$, $T\neq U,W$. 

If $T\sim U$ then its distances from the elements of $\Ls$ determine the unique $i\geq 3$ for which $T\in\ell_i$. For all $R\in\ell_i$, there is a unique point on $\ell$ collinear with $R$, and, since $i\neq 0$, these points are pairwise distinct. Thus $T$ is resolved by $S$. Suppose $T\not\sim U$. Note that there are $q^3$ such points. As for $i=0,1,2$, there is a unique point on $\ell_i$ collinear with $T$, so $T$ is the center of a unique triad $(x_0,x_i,x_2)$ with $x_i\in\ell_i$. Since there are $q^3$ such triads, each having as a center $U$ and thus, by Proposition \ref{triad}, another point not collinear with $U$, we see that any point $T\not\sim U$ is resolved by $S$, which finishes the proof. The calculation of the size is easy.

\end{proof}

\begin{corollary}

If $q$ is odd then the metric dimension of $W(q)$ is at most $8q-1$.

\end{corollary}

\begin{proof}

Take a semi-resolving set for the lines as in Proposition \ref{srsWqln}, and also one for the points as in Proposition \ref{srsWqpt} in such a way that $q-3$ collinear points of the former one is contained in the latter. Then the union of these two sets is a resolving set of size at most $8q-1$.

\end{proof}

We summarize the results of this section in the following theorem.

\begin{theorem}

The metric dimension of $W(q)$ satisfies the inequalities

\[\mathrm{max}\{6q-27,4q-7\} \leq \mu(W(q))\leq 8q.\]

\end{theorem}

\section*{Acknowledgement}

The research was partially supported by the Italian MIUR (progetto 40\% ``Strutture Geometriche, Combinatoria e loro Applicazioni''), GNSAGA, the bilateral Slovenian-Hungarian Joint Research Project no.\ NN 114614 (in Hungary) and N1-0032 (in Slovenia), and the Bolyai Research Grant.

\small

\begin{flushleft}

Daniele Bartoli\\

University of Perugia, Italy\\

Department of Mathematics and Informatics\\

via Vanvitelli 1, 06123, Perugia\\

e-mail: {\sf daniele.bartoli@unipg.it}

\end{flushleft}

\begin{flushleft}

Tam\'{a}s H\'{e}ger, Marcella Tak\'ats \\

MTA--ELTE Geometric and Algebraic Combinatorics Research Group\\
ELTE E\"otv\"os Lor\'and University, Budapest, Hungary\\
Department of Computer Science\\
1117 Budapest, P\'azm\'any P.\ stny.\ 1/C, Hungary\\
e-mail: {\sf hetamas@cs.elte.hu}, {\sf takats@cs.elte.hu}

\end{flushleft}

\begin{flushleft}

Gy\"orgy Kiss\\

Department of Geometry and\\

MTA--ELTE Geometric and Algebraic Combinatorics Research Group \\

ELTE E\"{o}tv\"{o}s Lor\'{a}nd University, Budapest, Hungary\\

1117 Budapest, P\'{a}zm\'{a}ny P\'eter s\'et\'any 1/C, Hungary, and \\

FAMNIT, University of Primorska, \\

6000 Koper, Glagolja\v ska 8, Slovenia \\

e-mail: {\sf kissgy@cs.elte.hu}\\
\end{flushleft}

\end{document}